\documentclass{article}
\usepackage{amssymb}
\usepackage{amsmath}
\usepackage{amsfonts}

\newtheorem{theorem}{Theorem}

\newtheorem{lemma}[theorem]{Lemma}

\newtheorem{proposition}[theorem]{Proposition}

\newenvironment{proof}[1][Proof]{\noindent\textbf{#1.} }{\ \rule{0.5em}{0.5em}}
\input{tcilatex}

\begin{document}

\author{Nadjia El Saadi $^{(1)}$\thanks{
Corresponding author. E-mail address: enadjia@gmail.com.} \& Zakia Benbaziz $%
^{(2)}$ \\
$^{(1)}$ LAMOPS, ENSSEA, Algiers, Algeria\\
$^{\text{ }(2)\ }$TPFA, ENS Kouba, Algiers, Algeria\\
}
\title{On the existence of solutions for a nonlinear stochastic partial
differential equation arising as a model of phytoplankton aggregation }
\date{}
\maketitle

\begin{abstract}
In this paper, we are interested in the analytical study of a nonlinear
Stochastic Partial Differential Equation (SPDE) arising as a model of
phytoplankton aggregation. This SPDE consists in a diffusion equation with a
chemotaxis term responsible of self-attraction of phytoplankton cells and a
multiplicative branching noise. Existence of mild solutions is established
through weak and tightness arguments.

\begin{equation*}
\end{equation*}%
\end{abstract}

\bigskip Introduction

In this paper, we are interested in the following nonlinear stochastic
partial differential equation:%
\begin{equation}
\frac{\partial }{\partial t}u(t,x)=D\frac{\partial ^{2}}{\partial x^{2}}%
u(t,x)-\frac{\partial }{\partial x}\left( u(t,x)\left[ G\ast u^{0}(t,.)%
\right] (x)\right) +\sqrt{\lambda u_{+}(t,x)}\overset{.}{W}(t,x),
\label{EDPS}
\end{equation}%
in $[0,T]\times \Omega ,$ where $\Omega =\left] 0,L\right[ $ is a bounded
domain with boundary $\partial \Omega $ in $\mathbb{R},\;x$ is a one
dimensional coordinate, $t$ is time. The motivation in studying this SPDE$\ $%
is that equation (\ref{EDPS}) arises as a model of phytoplankton
aggregation
(\cite{ElSaadi06},\cite{ElSaadi06Bis},\cite{ElSaadi07a},\cite{ElSaadi07}).
In fact, the authors in (\cite{ElSaadi06},\cite{ElSaadi06Bis},\cite%
{ElSaadi07a},\cite{ElSaadi07}) have investigated an individual-based model
(IBM) of a population of phytoplankton that takes into account small scales
biological mechanisms for phytoplankton cells. These processes are: random
dispersal of phytoplankton cells due to turbulence, spatial interactions
between phytoplankton cells caused by chemical signals and random division
and death of phytoplankton cells. The aim of such modelling was to capture
the main features of the individual dynamics at small scales that are
responsible at a larger scale for the formation of aggregating patterns. An
Eulerian version of the IBM has been derived in (\cite{ElSaadi06},\cite%
{ElSaadi07a}) and the SPDE $(\ref{EDPS})$ is obtained as a continuum limit
of this IBM. $u(t,x)$ represents the spatio-temporal distribution density of
phytoplankton on the vertical water column. We denote by $u_{+}(t,x)$ the
positive part of $u(t,x)$, that is $u_{+}(t,x)=\max (u(t,x),0).$ The
diffusion term in $(\ref{EDPS})$ takes into account the random spread of the
phytoplankton cells with the coefficient of diffusion $D$; while the
transport term, which is a chemotaxis term, describes the interaction
mechanisms between the phytoplankton cells via the velocity $G\ast u^{0}.$
The latter has the form of convolution as in \cite{Mogilner99}, i.e.,
\begin{equation*}
\left[ G\ast u^{0}(t,.)\right] (x)=\int_{\mathbb{R}}G\left( x-x^{\prime
}\right) u^{0}(t,x^{\prime })dx^{\prime },
\end{equation*}%
where $G$ is the attractive force given by:

\begin{equation*}
G\left( x\right) =\left[ \left( \left\vert x\right\vert -r_{1}\right) \left(
\left\vert x\right\vert -r_{0}\right) \right] 1_{\left[ -r_{1},-r_{0}\right]
\cup \left[ r_{0},r_{1}\right] }\left( x\right) .
\end{equation*}%
and
\begin{equation*}
u^{0}(x)=\left\{
\begin{array}{l}
u(x)\text{\ \ \ \ \ \ \ \ \ \ }0<x<L \\
0\;\;\;\ \ \ \ \ \ \ \ \ \ \ \ \ x\leq 0\;\text{or\ }x\geq L.%
\end{array}%
\right.
\end{equation*}%
The biological explanation of the interactions between phytoplankton cells
is based on the following works (\cite{Fitt84},\cite{Fitt85},\cite{Hauser},%
\cite{Levandowsky78},\cite{Levandowsky85},\cite{Spero79},\cite{Spero81},\cite%
{Spero85}) which report that species of phytoplankton such as
dinoflagellates and motile algae have chemosensory abilities i.e., they can
sense the chemical field generated by the presence of other dinoflagellates
and motile algae that are present at a certain distance. More precisely,
dinoflagellates and motile algae leak organic matter such as sugar and amino
acids, forming regions around them having concentrations higher than average
\cite{Azam}. Experiments studies on the chemosensory abilities in
dinoflagellates and motile algae (\cite{Spero85}, \cite{Fitt85}) show that
the released products attract other dinoflagellates and motile algae that
are present in a suitable neighborhood. It has also been observed that high
concentrations of these products inhibit the chemosensory behavior in
dinoflagellates and motile algae \cite{Fitt85}.\newline
So, if a phytoplankton cell is located at a position $x$, the extracellular
products released by this cell form a concentration around $x$, which is
highest in the closest vicinity of $x$ (for instance, on a radius of length $%
r_{0}$ ( $r_{0}$ small)) and then decreases progressively. Via their
chemosensory abilities, all dinoflagellates at positions $x^{\prime }$ such
that $r_{0}\leq \left\vert x-x^{\prime }\right\vert \leq r_{1}$ ( $%
r_{1}>r_{0}$ and $r_{0}$ too small relative to $r_{1}$) detect the
differences of concentration in water and hence are attracted to the cell in
$x.$ Beyond $r_{1}$, cells cannot perceive the difference in concentration
because they are sensory limited (\cite{Berg},\cite{Jackson87},\cite%
{Jackson89}). Hence, they are not attracted. $\left( G\ast u^{0}\right) (x)$
describes the velocity induced at the site $x$ by the net effects of all
phytoplankton cells at various sites $x^{\prime }$. The kernel $G\left(
x-x^{\prime }\right) $ associates a strength of interaction per unit density
as a function of the distance $x-x^{\prime }$ between any two phytoplankton
cells in sites $x$ and $x^{\prime }$. $G$ behaves as a gradient, that is, a
cell is attracted to the region of high density. \newline
$\overset{.}{W}(t,x)$ is a white noise in space and time defined on some
probability space $(\Lambda ,\mathcal{F},P)$ \cite{Walsh} and $\lambda $ is
the branching rate. The term $\sqrt{\lambda u_{+}(t,x)}\overset{.}{W}(t,x)$
describes the stochastic fluctuations of the number of phytoplankton cells
as a result of the random birth and death events.\newline
Here, $\mathbb{R}$ represents the vertical axis oriented downward from the
surface to the seabed. The point $0$ is at the surface of water and $L$ is
the limit of the "euphotic zone" ie. the upper layers in the vertical water
column. The restriction of model (\ref{EDPS}) to this domain is due to the
fact that phytoplankton cells can survive and multiply only in the
\textquotedblright euphotic zone\textquotedblright . Therefore, von-Neumann
boundary conditions are imposed at the surface $0$ and at $L$:
\begin{equation}
\dfrac{\partial }{\partial x}u(t,x)=0,\text{ on }\mathbb{[}0,T]\times
\partial \Omega .  \label{bord}
\end{equation}%
and the initial density is
\begin{equation}
u(0,x)=u_{0}(x)\geq 0,\text{ }in\text{ }\Omega .  \label{initial}
\end{equation}

\bigskip

To our knowledge, SPDE\ (\ref{EDPS}) is new and unknown in both biological
and mathematical literature. From the mathematical point of view, it is a
complication. The major difficulties come from the two nonlinear terms: the
chemotaxis term with a convolution and the stochastic branching term with
the nonlinear multiplicative noise. We point out that previous works (\cite%
{Adioui}, \cite{ElSaadi12}) have concerned the study of the deterministic
part of the SPDE (\ref{EDPS}) (i.e. equation (\ref{EDPS}) without the noise
term) with the goal of analyzing the influence of the spatial \ interactions
between phytoplankton cells on the aggregation process. Here in this paper,
we are interested by the whole equation (\ref{EDPS}) which includes the
effect of the stochastic branching term{\large . }The volatility $\sqrt{%
\lambda u_{+}(t,x)}$ of the space-time white noise is non-Lipschitzian;
hence, we are not able to expect the existence of strong or mild solutions
for (\ref{EDPS}). However, we might expect the existence of a mild solution
to (\ref{EDPS}) in the weak sense; that is we shall construct a probability
basis $(\overline{\Lambda },\overline{\mathcal{F}},(\overline{\mathcal{F}}%
_{t})_{t\in \lbrack 0,T]},\overline{P})$, on which there exists a Gaussian
space-time white noise $\overline{\overset{.}{W}}(t,x)$ and a mild solution $%
\widetilde{u}(t,x)$ to (\ref{EDPS}).{\large \ }This will be done by using
tightness arguments which are useful in obtaining weak convergence. The rest
of the paper is organized as follows. In the next section, we present the
abstract formulation of the problem and in section 3, we present our results
on existence of mild solutions at the weak sense. Several steps are arranged
for proving the main theorem. Some estimates on the semi-group generated by
the operator $\dfrac{\partial ^{2}}{\partial x^{2}}$ are listed in the
Appendix.

\section{Abstract formulation}

We reformulate problem (\ref{EDPS}) with conditions (\ref{bord}) and (\ref%
{initial}) as a stochastic version of an abstract Cauchy problem which can
be treated by using the theory of semigroups of operators:%
\begin{equation}
\left\{
\begin{array}{ll}
du(t) & =\left( Au(t)-B\left[ u(t)g_{G}\left( u(t)\right) \right] \right)
dt+C(u(t))dW(t) \\
u(0) & =u_{0}%
\end{array}%
\right.  \label{8}
\end{equation}%
The operator $A:$ ${\mathcal{D}}(A)\subset X:$ $=L^{2}(\Omega )\rightarrow X$
is defined by
\begin{equation}
\begin{array}{ll}
Aw & =D\dfrac{d^{2}w}{dx^{2}}, \\
{\mathcal{D}}(A) & =\left\{ w\in H^{2}(\Omega ):\text{ }w_{\left\vert
\partial \Omega \right. }^{\prime }=0\right\} ,%
\end{array}
\label{10000}
\end{equation}%
and the operator $B:$ ${\mathcal{D}}(B)\subset X\rightarrow X$ by%
\begin{eqnarray}
Bw &=&\dfrac{d}{dx}w,  \label{10002} \\
{\mathcal{D}}(B) &=&H^{1}(\Omega ).  \notag
\end{eqnarray}%
$H^{1}(\Omega )$ and $H^{2}(\Omega )$ denote usual Sobolev functions spaces$%
. $ We will denote by $\langle ,\rangle $ and $\left\Vert .\right\Vert ,$
respectively, the scalar product and the norm in $X$. The operator $A$
commutes with $B$ and they are related by the following formula
\begin{equation}
\left\langle Bu,DBu\right\rangle =-\left\langle u,Au\right\rangle ,\forall
u\in {\mathcal{D}}(A).  \label{tlemcen1}
\end{equation}%
We endow $D(B)$ with the graph norm $\left\vert x\right\vert _{{\mathcal{D}}%
(B)}=\left\Vert Bx\right\Vert +$\ $\left\Vert x\right\Vert $ for $x\in D(B).$%
\newline
The operator $g_{G}$ is defined as follows:%
\begin{equation}
g_{G}\left( \varphi \right) \left( x\right) =\left[ G\ast \varphi \right]
\left( x\right) =\dint_{\mathbb{R}}G\left( x-y\right) \varphi \left(
y\right) dy.  \label{conv}
\end{equation}
\newline
By straightforward consequence of standard calculations, we can establish
that $g_{G}:{\mathcal{D}}(B)\rightarrow {\mathcal{D}}(B)$, continuously so
there exists a constant $\delta ,$ so that
\begin{equation}
\left\vert g_{G}\left( \varphi \right) \right\vert _{{\mathcal{D}}\left(
B\right) }\leq \delta \left\vert \varphi \right\vert _{{\mathcal{D}}\left(
B\right) },\text{\thinspace }\forall \varphi \in {\mathcal{D}}(B).
\label{cont}
\end{equation}%
Note also that $G\ast u^{0}$ is uniformly bounded. Hence, $g_{G}(u)$ is
uniformly bounded. As a result of H\"{o}lder's inequality, we get
\begin{equation}
\left\vert g_{G}(u)\right\vert _{\infty }\leq \sqrt{L}\left\vert
G\right\vert _{\infty }\left\Vert u\right\Vert ,\text{\thinspace }\forall
u\in D(B).  \label{hauston1}
\end{equation}%
On the other hand, we have
\begin{equation}
\left\vert Bg_{G}(u)\right\vert _{\infty }\leq \sqrt{L}\left\vert
G\right\vert _{\infty }\left\vert u\right\vert _{{\mathcal{D}}\left(
B\right) },\text{\thinspace }\forall u\in D(B).  \label{hauston2}
\end{equation}

\bigskip $C$ is the non linear operator,
\begin{equation*}
C:X\longrightarrow L_{HS}({\mathcal{D}}(B),X)
\end{equation*}%
where $C(u)$ is the linear operator of multiplication by the function $\sqrt{%
\lambda u_{+}},$ that is$:$
\begin{equation*}
\forall u\in X,\text{ }w\in {\mathcal{D}}(B),\quad (C(u)\text{ }w)(x)=\sqrt{%
\lambda u_{+}(x)}w(x),\qquad x\in \Omega
\end{equation*}%
$L_{HS}({\mathcal{D}}(B),X)${\Large \ }denotes the space of Hilbert-Schmidt
operators from ${\mathcal{D}}(B)$\ to $X$\ equipped with the norm%
\begin{equation*}
\left\Vert T\right\Vert _{2}=(\sum\limits_{k=1}^{\infty
}\left\Vert Te_{k}\right\Vert ^{2})^{1/2}\ <{\infty },
\newline
\end{equation*}%
where $\{e_{k}\}$\ is a complete orthonormal basis in $X.$
\newline
Suppose that the compact embedding: \newline
\begin{equation}
J:{\mathcal{D}}(B)\hookrightarrow X\text{ \ \ \ \ \ \ \ \ is Hilbert-Schmidt,%
}  \label{cond}
\end{equation}%
then $(W(t))_{t\in \lbrack 0,T]}$ can be considered as a cylindrical Wiener
process on ${\mathcal{D}}(B)$ with values in $X,$ defined on the probability
space $(\Lambda ,\mathcal{F},P)$ with a filtration $(\mathcal{F}_{t})_{t\in
\lbrack 0,T]}$.

\begin{proposition}
\label{analytique}The operator $A$ defined by $\left( \ref{10000}\right) $
is the generator of an analytic semigroup of contractions in $X$, $%
(T(t))_{t\geq 0},$ compact for $t>0$. The restrictions $T(t)_{/{\mathcal{D}}%
(B)}$ send ${\mathcal{D}}(B)$ into itself and are uniformly bounded in ${%
\mathcal{D}}(B)$ (that is, there exists $C_{1}\geq 0,$ such that, $%
\left\vert T(t)_{/{\mathcal{D}}(B)}\right\vert _{D(B)}\leq C_{1}$, for $%
t\geq 0$).
\end{proposition}

\begin{proof}
The proof is similar to that one in \cite{Adioui}.
\end{proof}

We propose to solve, on time interval $[0,T],$ equation {\large (}\ref{8}%
{\large ) }in integrated form by using the stochastic generalization of the
classical variation of constants formula%
\begin{equation}
\begin{array}{l}
u(t)=T(t)u_{0}-\dint_{0}^{t}T(t-s)B\left[ u(s)g_{G}\left( u(s)\right) \right]
ds+\dint_{0}^{t}T(t-s)C\left( u(s)\right) dW(s).%
\end{array}
\label{6}
\end{equation}%
We remind that a predictable $X$-valued stochastic process $(u(t))_{t\in
\lbrack 0,T]}$ defined on the probability space $(\Lambda ,\mathcal{F},(%
\mathcal{F}_{t})_{t\in \lbrack 0,T]},P)$ is called a mild solution of the
differential equation (\ref{8}) if $u(t)$ is a solution of (\ref{6}). On
another hand, the stochastic process $(u(t))_{t\in \lbrack 0,T]}$ is called
a weakened solution of (\ref{8}) if $u(t)$ is a solution of
\begin{equation}
u(t)=u_{0}+A(\dint\limits_{0}^{t}u(s)ds)-\dint\limits_{0}^{t}B\left[
u(s)g_{G}\left( u(s)\right) \right] ds+\dint\limits_{0}^{t}C(u(s))dW(s).
\label{wkned}
\end{equation}%
It is well known that solutions of (\ref{6}) and (\ref{wkned}) are
equivalent (\cite{CH},\cite{CH1}). Note that the stochastic integral $%
\dint_{0}^{t}T(t-s)C\left( u(s)\right) dW(s)$ is well defined since $%
E(\dint_{0}^{t}\left\Vert T(t-s)C\left( u(s)\right) \right\Vert
_{2}^{2}ds)<\infty ,\forall t\in \lbrack 0,T]$ . This is due to the fact
that $C\left( u(s)\right) $ is Hilbert-Schmidt operator and $T(t-s)$ is
linear bounded operator then, basing on the theory of Hilbert-Schmidt
operators (e.g. \cite{Gelfand}, Chapter\textit{\ }I\textit{), }$%
T(t-s)C\left( u(s)\right) $ is Hilbert-Schmidt operator too$.$\newline

\section{Existence of solutions}

This section is concerned with the existence of mild solutions for (\ref{8}%
). For this purpose, we need first to establish several propositions and
lemmas. We start by giving some useful estimates.

\begin{lemma}
\label{prin}1) There exists a constant $M,$ such that, for all $u$, $v\in {%
\mathcal{D}}\left( B\right) ,$ we have
\begin{equation*}
\left\Vert B\left[ ug_{G}(u)\right] -B\left[ vg_{G}(v)\right] \right\Vert
\leq M\max (\left\vert u\right\vert _{{\mathcal{D}}\left( B\right) }\text{, }%
\left\vert v\right\vert _{{\mathcal{D}}\left( B\right) })\left\vert
u-v\right\vert _{{\mathcal{D}}\left( B\right) }.
\end{equation*}%
2) There exists a positive constant $Q,$ such that, for all $u\in {\mathcal{D%
}}\left( B\right) ,$ it holds that
\begin{equation*}
\left\Vert B\left[ ug_{G}(u)\right] \right\Vert \leq Q\left\vert
u\right\vert _{{\mathcal{D}}\left( B\right) }\left\Vert u\right\Vert .
\end{equation*}%
3) There exists a positive constant $C,$ such that, for all $u\in X,$ it
holds that
\begin{equation}
\left\Vert BT(t)u\right\Vert \leq \frac{C}{\sqrt{t}}\left\Vert u\right\Vert
,\;\forall t>0.  \label{rouen4}
\end{equation}
\end{lemma}

\begin{proof}
For the claim, the proof is the same as the one given in (Lemma 2.1, \cite%
{Adioui0}).
\end{proof}

By using Lemma 2, we can prove

\begin{proposition}
\label{prop1}Consider the following abstract problem%
\begin{equation}
\left\{
\begin{array}{ll}
du(t) & =\left( Au(t)-B\left[ u(t)g_{G}\left( u(t)\right) \right] \right)
dt+a(u(t))dW(t) \\
u(0) & =u_{0}.%
\end{array}%
\right.  \label{pb1}
\end{equation}%
Suppose that the operator $a$ satisfies%
\begin{equation*}
a:X\rightarrow L_{HS}({\mathcal{D}}\left( B\right) ,X)\text{ }
\end{equation*}%
and there exists a constant $K$ such that for all $x,y\in X,$%
\begin{equation*}
\left\Vert a(x)-a(y)\right\Vert _{2}\leqslant K\left\Vert x-y\right\Vert .
\end{equation*}%
Then, for each $u_{0}$ $\in D(B)$, problem (\ref{pb1}) has a unique mild
solution $u(t)$, $t\in \lbrack 0,T]$ \ such that $\underset{t\in \lbrack 0,T]%
}{sup}E(\left\vert u(t)\right\vert _{D(B)}^{2})<\infty .$
\end{proposition}

\begin{proof}
We use the method of successive approximations. Let $u_{0}$ $\in D(B)$ and
define the sequence $(u_{n})_{n\geq 1}$ by%
\begin{equation*}
\begin{array}{l}
u_{n+1}(t)=T(t)u_{0}(t)-\dint%
\limits_{0}^{t}T(t-s)B[u_{n}(s)g_{G}(u_{n}(s))]ds+\dint%
\limits_{0}^{t}T(t-s)a(u_{n}(s))dW(s).%
\end{array}%
\end{equation*}%
Let us denote by $Y$ the Banach space of all nonanticipating $D(B)$-valued
continuous stochastic processes $\left\{ U(t)\right\} _{t\in \lbrack 0,T]}$
endowed with the norm%
\begin{equation*}
\left\vert \left\vert \left\vert U\right\vert \right\vert \right\vert
=\left\{ \underset{0\leq t\leq T}{\sup }E\left\vert U(t)\right\vert
_{D(B)}^{2}\right\} ^{1/2}<\infty
\end{equation*}%
Assume that $(u_{n}(t))_{n\geq 1}$ is bounded in $D(B),$ that is, $\exists
R>0$ such that $\left\vert (u_{n}(t)\right\vert _{D(B)}\leq R$, $\forall
t\in \lbrack 0,T],\forall n\geq 1.$ Hence $(u_{n}(t),t\in \lbrack
0,T])_{n\geq 1}$ is also bounded in $Y$, since $\left\vert \left\vert
\left\vert u_{n}\right\vert \right\vert \right\vert =\left\{ \underset{0\leq
t\leq T}{\sup }E\left\vert u_{n}(t)\right\vert _{D(B)}^{2}\right\}
^{1/2}\leq R.$

Let
\begin{equation*}
h_{n+1}(t)=E(\left\vert u_{n+1}(t)-u_{n}(t)\right\vert _{D(B)}^{2}),\quad
n\geq 0
\end{equation*}%
We have%
\begin{equation*}
\begin{array}{l}
h_{n+1}(t)\leq 2E\left\vert \dint\limits_{0}^{t}T(t-s)\left[
a(u_{n}(s))-a(u_{n-1}(s))\right] dW(s)\right\vert _{D(B)}^{2} \\
+2E\left\vert \dint\limits_{0}^{t}T(t-s)\left(
B[u_{n}(s)g_{G}(u_{n}(s))]-B[u_{n-1}(s)g_{G}(u_{n-1}(s))]\right)
ds\right\vert _{D(B)}^{2} \\
\leq \underset{I}{\underbrace{2E(\dint\limits_{0}^{t}\left\Vert T(t-s)\left[
a(u_{n}(s))-a(u_{n-1}(s))\right] \right\Vert _{2}^{2}ds)}} \\
+\underset{II}{\underbrace{2E\left\vert \dint\limits_{0}^{t}T(t-s)\left(
B[u_{n}(s)g_{G}(u_{n}(s))]-B[u_{n-1}(s)g_{G}(u_{n-1}(s))]\right)
ds\right\vert _{D(B)}^{2}}}%
\end{array}%
\end{equation*}

Since $\exists \ C_{1}\geq 0$ such that $\left\vert T(t)_{/{\mathcal{D}}%
(B)}\right\vert _{D(B)}\leq C_{1}$, for $t\geq 0,$ hence%
\begin{equation*}
\begin{array}{cc}
I & \leq 2C_{1}^{2}E(\dint\limits_{0}^{t}\left\Vert
a(u_{n}(s))-a(u_{n-1}(s))\right\Vert _{2}^{2}ds) \\
& \leq 2C_{1}^{2}K^{2}E(\dint\limits_{0}^{t}\left\Vert
u_{n}(s)-u_{n-1}(s)\right\Vert ^{2}ds) \\
& \leq 2C_{1}^{2}K^{2}T\ E(\underset{0\leq s\leq T}{\sup }\left\vert
u_{n}(s)-u_{n-1}(s)\right\vert _{D(B)}^{2}).%
\end{array}%
\end{equation*}

To calculate $II$, we have%
\begin{equation*}
\begin{array}{l}
\left\vert \dint\limits_{0}^{t}T(t-s)\left(
B[u_{n}(s)g_{G}(u_{n}(s))]-B[u_{n-1}(s)g_{G}(u_{n-1}(s))]\right)
ds\right\vert _{D(B)} \\
=\left\Vert \dint\limits_{0}^{t}T(t-s)\left(
B[u_{n}(s)g_{G}(u_{n}(s))]-B[u_{n-1}(s)g_{G}(u_{n-1}(s))]\right)
ds\right\Vert \\
+\left\Vert B\dint\limits_{0}^{t}T(t-s)\left(
B[u_{n}(s)g_{G}(u_{n}(s))]-B[u_{n-1}(s)g_{G}(u_{n-1}(s))]\right)
ds\right\Vert \\
\leq \dint\limits_{0}^{t}M\max (\left\vert u_{n}(s)\right\vert
_{D(B)},\left\vert u_{n-1}(s)\right\vert _{D(B)})\left\vert
u_{n}(s)-u_{n-1}(s)\right\vert _{D(B)}ds \\
+\dint\limits_{0}^{t}\dfrac{C}{\sqrt{t-s}}\left\Vert
B[u_{n}(s)g_{G}(u_{n}(s))]-B[u_{n-1}(s)g_{G}(u_{n-1}(s))]\right\Vert ds \\
\leq MR(T\ +2\sqrt{T}C)\underset{0\leq t\leq T}{\sup }\left\vert
u_{n}(t)-u_{n-1}(t)\right\vert _{D(B)}.%
\end{array}%
\end{equation*}%
Therefore%
\begin{equation*}
II\leq 2M^{2}R^{2}(T\ +2\sqrt{T}C)^{2}E(\underset{0\leq t\leq T}{\sup }%
\left\vert u_{n}(t)-u_{n-1}(t)\right\vert _{D(B)}^{2}).
\end{equation*}%
Hence, we obtain%
\begin{equation*}
\begin{array}{cc}
h_{n+1}(t) & \leq 2[M^{2}R^{2}(\ T+2\sqrt{T}C)^{2}+C_{1}^{2}K^{2}T]E(%
\underset{0\leq t\leq T}{\sup }\left\vert u_{n}(t)-u_{n-1}(t)\right\vert
_{D(B)}^{2})%
\end{array}%
\end{equation*}%
By choosing $T>0$ small enough so that
\begin{equation}
2[M^{2}R^{2}(\ T+2\sqrt{T}C)^{2}+C_{1}^{2}K^{2}T]<\frac{1}{2},
\label{condition T}
\end{equation}%
we obtain%
\begin{equation*}
h_{n+1}(t)\leq \frac{1}{2^{n}}E(\underset{0\leq t\leq T}{\sup }\left\vert
u_{1}(t)-u_{0}(t)\right\vert _{D(B)}^{2}).
\end{equation*}%
$u_{n+1}(t)-u_{n}(t)$ is the general term of an absolutely convergent series
in the Banach space $Y.$ Hence, $\{u_{n}(.)\}$ is a convergent sequence in $%
Y.$ The limit in $Y,$ $u_{\infty }(t)=\underset{n\longrightarrow \infty }{%
\lim }u_{n}(t)$ is a solution of (\ref{pb1}) for fixed $t$. To complete the
proof of the proposition, we have to show that the sequence $u_{n}(t)$
remains in $Y,$ \ i.e. $E(\underset{0\leq t\leq T}{\sup }\left\vert
u_{n+1}(t)\right\vert _{D(B)}^{2})<\infty .$%
\begin{equation*}
\begin{array}{ll}
E(\left\vert u_{n+1}(t)\right\vert _{D(B)}^{2}) & \leq 2C_{1}^{2}\left\vert
u_{0}(t)\right\vert _{D(B)}^{2}+2M^{2}R^{2}(T+2C\sqrt{T})^{2} \\
& E(\underset{0\leq t\leq T}{\sup }\left\vert u_{n}(t)\right\vert
_{D(B)}^{2})+2C_{1}^{2}K_{1}^{2}T\ E(\underset{0\leq s\leq T}{\sup }%
\left\vert u_{n}(s)\right\vert _{D(B)}^{2}) \\
& \leq 2C_{1}^{2}R^{2}+2R^{2}\left[ M^{2}R^{2}(T+2C\sqrt{T}%
)^{2}+C_{1}^{2}K_{1}^{2}T\right] ,\forall t\in \lbrack 0,T].%
\end{array}%
\end{equation*}%
Using $T>0$ such that (\ref{condition T}) holds, we have%
\begin{equation*}
\underset{0\leq t\leq T}{\sup }E(\left\vert u_{n+1}(t)\right\vert
_{D(B)}^{2})\leq 2C_{1}^{2}R^{2}+\frac{1}{2}R^{2}<\infty \text{ \ .\ }
\end{equation*}%
To prove the uniqueness of the solution (up to equivalence), consider $u$
and $v$ two mild solutions of (\ref{pb1}) on $[0,T].$ We might show that
\begin{equation*}
\underset{0\leq t\leq T}{\sup }E\left\vert u(t)-v(t)\right\vert
_{D(B)}^{2}=0.\newline
\end{equation*}%
We have%
\begin{equation*}
\begin{array}{cc}
E\left\vert u(t)-v(t)\right\vert _{D(B)}^{2} & \leq \underset{(1)}{%
\underbrace{2E\left\vert
\dint\limits_{0}^{t}T(t-s)(B[u(s)g_{G}(u(s))]-B[v(s)g_{G}(v(s))])ds\right%
\vert _{D(B)}^{2}}} \\
& +\underset{(2)}{\underbrace{2E\left\vert
\dint\limits_{0}^{t}T(t-s)[a(u(s))-a(v(s))]dW(s)\right\vert _{D(B)}^{2}}}%
\end{array}%
\end{equation*}%
We obtain%
\begin{equation*}
(2)\leq 2C_{1}^{2}K^{2}TE(\underset{0\leq s\leq T}{\sup }\left\vert
u(s)-v(s)\right\vert _{D(B)}^{2})
\end{equation*}%
To calculate $(1)$, we have%
\begin{equation*}
\begin{array}{cc}
\left\vert
\dint\limits_{0}^{t}T(t-s)(B[u(s)g_{G}(u(s))]-B[v(s)g_{G}(v(s))])ds\right%
\vert _{D(B)} &  \\
\leq MR(T+2C\sqrt{T})\underset{0\leq t\leq T}{\sup }\left\vert
u(t)-v(t)\right\vert _{D(B)} &
\end{array}%
\end{equation*}%
Then%
\begin{equation*}
(1)\leq 2M^{2}\gamma ^{2}R^{2}(T+2C\sqrt{T})^{2}E(\underset{0\leq t\leq T}{%
\sup }\left\vert u(t)-v(t)\right\vert _{D(B)}^{2})
\end{equation*}%
Finally, we obtain%
\begin{equation*}
\begin{array}{ll}
E\left\vert u(t)-v(t)\right\vert _{D(B)}^{2} & \leq \left[ 2M^{2}R^{2}(T+2C%
\sqrt{T})^{2}+2C_{1}^{2}K^{2}T\right] \\
& E(\underset{0\leq t\leq T}{\sup }\left\vert u(t)-v(t)\right\vert
_{D(B)}^{2})\text{ \ \ \ \ \ \ \ \ \ }\forall t\in \lbrack 0,T]%
\end{array}%
\end{equation*}%
Since $T$ is chosen such that $\left[ 2M^{2}R^{2}(T+2C\sqrt{T}%
)^{2}+2C_{1}^{2}K^{2}T\right] <\frac{1}{2}$, hence we have%
\begin{equation*}
\underset{0\leq t\leq T}{\sup }E(\left\vert u(t)-v(t)\right\vert
_{D(B)}^{2})=0\text{\ \ \ \ }
\end{equation*}%
which leads to%
\begin{equation*}
u(t)=v(t),\ \forall t\in \lbrack 0,T],\qquad P-a.s.\text{\ \ \ }
\end{equation*}
\end{proof}

We return to problem (\ref{8}). We propose as in \cite{Bo} to construct a
sequence of Lipschitz continuous mappings $(a_{n}(u))_{n\in
\mathbb{N}
}$ to approximate the non-Lipschitzian function$\sqrt{\lambda u^{+}}$ on $%
\mathbb{R}
.$ For $n\in
\mathbb{N}
$, define $a_{n}(.)$ by%
\begin{equation}
a_{n}(u)=\left\{
\begin{array}{cc}
0 & u<0 \\
\sqrt{n}u & \ \ 0\leq u<\dfrac{1}{n} \\
\sqrt{\lambda u} & u\geq \dfrac{1}{n}%
\end{array}%
\right.  \label{contr1}
\end{equation}%
It is clear that the sequence $(a_{n}(u))_{n}$ converges to $\sqrt{\lambda
u^{+}}$ uniformly on $u$ $\in $ $%
\mathbb{R}
$ as $n\longrightarrow +\infty .$

\begin{lemma}
Suppose that the sequence $(a_{n}(u))_{n}$ defined by (\ref{contr1})
satisfies
\begin{equation}
a_{n}:X\rightarrow L_{HS}(\mathcal{D(}B\mathcal{)},X)\text{ }  \label{contr2}
\end{equation}%
so that%
\begin{equation}
\forall u\in X,\text{ }w\in {\mathcal{D}}(B),\quad a_{n}(u)\text{ }%
w=a_{n}(u)w.  \label{contr3}
\end{equation}%
Then, there exists a constant $K^{\prime }$ such that for all $u,v\in X,$%
\begin{equation}
\left\Vert a_{n}(u)-a_{n}(v)\right\Vert _{2}\leqslant K^{\prime }\left\Vert
u-v\right\Vert .  \label{contr4}
\end{equation}
\end{lemma}

\begin{proof}
Let $u,v\in X$ and $\{e_{k}\}$ a complete orthonormal basis in $D(B).$ Hence%
\begin{equation*}
\begin{array}{ll}
\left\Vert a_{n}(u)-a_{n}(v)\right\Vert _{2}^{2} & =\sum\limits_{k=1}^{+%
\infty }\left\Vert (a_{n}(u)-a_{n}(v))e_{k}\right\Vert ^{2} \\
& =\sum\limits_{k=1}^{+\infty }\int\limits_{0}^{L}\left\vert
a_{n}(u(x))-a_{n}(v(x))\right\vert ^{2}\left\vert e_{k}(x)\right\vert ^{2}dx%
\end{array}%
\end{equation*}%
As $a_{n}(u)$ is lipschitz in $u$ on $%
\mathbb{R}
,$ then there exists $K>0$ such that
\begin{equation*}
\begin{array}{ll}
\left\Vert a_{n}(u)-a_{n}(v)\right\Vert _{2}^{2} & \leq
K^{2}\sum\limits_{k=1}^{+\infty }\int\limits_{0}^{L}\left\vert
u(x))-v(x)\right\vert ^{2}\left\vert e_{k}(x)\right\vert ^{2}dx \\
& \leq K^{2}\sum\limits_{k=1}^{+\infty }\left\Vert (u-v)e_{k}\right\Vert ^{2}
\\
& \leq K^{2}\left\Vert u-v\right\Vert ^{2}\sum\limits_{k=1}^{+\infty
}\left\vert e_{k}\right\vert _{D(B)}^{2}%
\end{array}%
\end{equation*}%
Using the hypothesis (\ref{cond}), we have%
\begin{equation*}
\sum\limits_{k=1}^{+\infty }\left\vert e_{k}\right\vert
_{D(B)}^{2}=\sum\limits_{k=1}^{+\infty }\left\Vert Je_{k}\right\Vert
^{2}=\left\Vert J\right\Vert _{2}^{2}<\infty ,
\end{equation*}%
which leads to the fact that there exists $M_{1}>0$ such that
\begin{equation*}
\left\Vert a_{n}(u)-a_{n}(v)\right\Vert _{2}^{2}\leq M_{1}K^{2}\left\Vert
u-v\right\Vert ^{2}
\end{equation*}
\end{proof}

\bigskip

Let us now consider the following approximating SPDE's:

\ \

\begin{equation}
u_{n}(t)=I_{0}(t)-I_{n}(t)+J_{n}(t),\text{ \ \ \ \ }n\in
\mathbb{N}
\label{approx}
\end{equation}

with%
\begin{equation*}
\begin{array}{l}
I_{0}(t)=T(t)u_{0}(t) \\
I_{n}(t)=\dint\limits_{0}^{t}T(t-s)B[u_{n}(s)g_{G}(u_{n}(s))]ds \\
J_{n}(t)=\dint\limits_{0}^{t}T(t-s)a_{n}(u_{n}(s))dW(s)%
\end{array}%
\end{equation*}

where the sequence $(a_{n}(.))_{n}$ is defined by (\ref{contr1})\ and
satisfies (\ref{contr2}).

We aim to prove weak convergence of the sequence $\{u_{n}(t),t\in \lbrack
0,T]\}_{n\in
\mathbb{N}
}$ by proving the tightness of $\{I_{n}(t),t\in \lbrack 0,T]\}_{n\in
\mathbb{N}
}$ and $\{J_{n}(t),t\in \lbrack 0,T]\}_{n\in
\mathbb{N}
}$ since $I_{0}(t)$ is deterministic.

\begin{proposition}
Let the initial condition $u_{0}\in D(B)$ be a continuous deterministic
mapping on $\Omega .$The sequences $\{I_{n}(t),t\in \lbrack 0,T]\}_{n\in
\mathbb{N}
}$ and $\{J_{n}(t),t\in \lbrack 0,T]\}_{n\in
\mathbb{N}
}$ are tight in $\mathcal{C(}[0,T],D(B)).$
\end{proposition}

\begin{proof}
By Lemma 4, it follows that for each $n\in
\mathbb{N}
,$ the sequence $(a_{n}(.))_{n}$ satisfies the conditions of Proposition \ref%
{prop1}. Hence, there exists a unique mild solution $u_{n}(t),t\in \lbrack
0,T]$ for problem (\ref{approx}) such that $\underset{t\in \lbrack 0,T]}{sup}%
E(\left\vert u_{n}(t)\right\vert _{D(B)}^{2})<\infty .$ For given $0\leq
t^{\prime }\leq t\leq T$ and \ $\eta \in \lbrack 0,\frac{1}{2}],$ it follows
from Lemma $\ref{A1}$ (in Appendix A) that%
\begin{equation*}
\begin{array}{ll}
E\left\vert J_{n}(t)-J_{n}(t^{\prime })\right\vert _{D(B)}^{2} & \leq
2E(\dint\limits_{0}^{t}\left\Vert (T(t-s)-T(t^{\prime
}-s))a_{n}(u_{n}(s))\right\Vert _{2}^{2}ds) \\
& +2E(\dint\limits_{t}^{t^{\prime }}\left\Vert (T(t^{\prime
}-s)a_{n}(u_{n}(s))\right\Vert _{2}^{2}ds) \\
& \leq 2C^{\prime }K^{2}E(\underset{0\leq s\leq T}{\sup }\left\vert
u_{n}(s)\right\vert _{D(B)}^{2})\left\vert t-t^{\prime }\right\vert ^{\eta }
\\
& +2E(\dint\limits_{t}^{t^{\prime }}\left\vert B(T(t^{\prime }-s)\right\vert
_{D(B)}^{2}\left\vert u_{n}(s)g_{G}(u_{n}(s))\right\vert _{D(B)}^{2}ds \\
& \leq \mathcal{C}_{1}\left\vert t-t^{\prime }\right\vert ^{\eta }\text{ .\
\ \ }%
\end{array}%
\end{equation*}

On another hand, by using Lemma $\ref{A2}$ (in Appendix A), we have for
given $0\leq t^{\prime }\leq t\leq T$ and \ $\eta \in \lbrack 0,\frac{1}{2}]$%
\begin{equation*}
\begin{array}{ll}
E\left\vert I_{n}(t)-I_{n}(t^{\prime })\right\vert _{D(B)}^{2} & \leq
2E(\dint\limits_{0}^{t}\left\vert (T(t-s)-T(t^{\prime
}-s))B[u_{n}(s)g_{G}(u_{n}(s))]\right\vert _{D(B)}^{2}ds \\
& +2E(\dint\limits_{t}^{t^{\prime }}\left\vert B(T(t^{\prime
}-s)[u_{n}(s)g_{G}(u_{n}(s))]\right\vert _{D(B)}^{2} \\
& \leq 2E(\dint\limits_{0}^{t}\left\vert B(T(t-s)-T(t^{\prime
}-s))\right\vert _{D(B)}^{2}\left\vert u_{n}(s)g_{G}(u_{n}(s))\right\vert
_{D(B)}^{2}ds \\
& +2E(\dint\limits_{t}^{t^{\prime }}\left\vert B(T(t^{\prime }-s)\right\vert
_{D(B)}^{2}\left\vert u_{n}(s)g_{G}(u_{n}(s))\right\vert _{D(B)}^{2}ds%
\end{array}%
\end{equation*}%
using (\ref{cont}), we obtain%
\begin{equation*}
\begin{array}{ll}
E\left\vert I_{n}(t)-I_{n}(t^{\prime })\right\vert _{D(B)}^{2} & \leq
2C^{\prime \prime }\delta ^{2}E(\underset{0\leq s\leq T}{\sup }\left\vert
u_{n}(s)\right\vert _{D(B)}^{4})\left\vert t-t^{\prime }\right\vert ^{\eta }
\\
& +2C_{1}^{^{\prime \prime }}\delta ^{2}E(\underset{0\leq s\leq T}{\sup }%
\left\vert u_{n}(s)\right\vert _{D(B)}^{4})\left\vert t-t^{\prime
}\right\vert ^{\eta } \\
& \leq \mathcal{C}_{2}\left\vert t-t^{\prime }\right\vert ^{\eta }\text{ .}%
\end{array}%
\end{equation*}%
It follows that
\begin{equation*}
\begin{array}{ll}
E\dint\limits_{0}^{T}\dint\limits_{0}^{T}\left( \dfrac{\left\vert
J_{n}(t)-J_{n}(t^{\prime })\right\vert _{D(B)}}{\left\vert t-t^{\prime
}\right\vert ^{\gamma }}\right) ^{2}dtdt^{\prime } & \leq
\dint\limits_{0}^{T}\dint\limits_{0}^{T}\mathcal{C}_{1}\left\vert
t-t^{\prime }\right\vert ^{\eta -2\gamma }dtdt^{\prime } \\
& <\infty \text{ \ \ \ \ \ \ \ \ \ \ \ \ for all \ }0<\gamma <\frac{\eta +1}{%
2}%
\end{array}%
\end{equation*}%
For $R>0$, we define $A_{R}\subset \Lambda $%
\begin{equation*}
A_{R}=\left\{ \omega \in \Lambda \
/\dint\limits_{0}^{T}\dint\limits_{0}^{T}\left( \frac{\left\vert
J_{n}(t,\omega )-J_{n}(t^{\prime },\omega )\right\vert _{D(B)}}{\left\vert
t-t^{\prime }\right\vert ^{\gamma }}\right) ^{2}dtdt^{\prime }\leq R\right\}
\end{equation*}%
If we let in the Garsia-Rodemich-Rumsey Lemma (Lemma 9, in Appendix A): $%
f=J_{n}(.,w),$ $x=t,$ $y=t^{\prime },$ $p(x-y)=\left\vert x-y\right\vert
^{\gamma }$ $ie$ $\ p(x)=\left\vert x\right\vert ^{\gamma }$ and $\psi
(x)=x^{2},$ then%
\begin{equation*}
\left\Vert J_{n}(t,\omega )-J_{n}(t^{\prime },\omega \right\Vert _{D(B)}\leq
8\dint\limits_{0}^{\left\vert t-t^{\prime }\right\vert }\psi ^{-1}(\frac{4R}{%
u^{2}})dp(u).
\end{equation*}%
Since $\psi ^{-1}(\frac{4R}{u^{2}})=$ $\underset{v^{2}\leq \frac{4R}{u^{2}}}{%
\sup v}$ $=2\dfrac{\sqrt{R}}{u},$ thus%
\begin{equation*}
\begin{array}{ll}
\left\vert J_{n}(t,\omega )-J_{n}(t^{\prime },\omega \right\vert _{D(B)} &
\leq 8\dint\limits_{0}^{\left\vert t-t^{\prime }\right\vert }2\frac{\sqrt{R}%
}{u}\gamma u^{\gamma -1}du \\
& \leq 16\ \ \dfrac{\gamma }{\gamma -1}\ R^{\frac{1}{2}\ }\left\vert
t-t^{\prime }\right\vert ^{\gamma -1}.%
\end{array}%
\end{equation*}%
Finally
\begin{equation*}
\left\vert J_{n}(t,\omega )-J_{n}(t^{\prime },\omega \right\vert
_{D(B)}^{2}\leq M^{\prime }R\left\vert t-t^{\prime }\right\vert ^{2\gamma
-2}.
\end{equation*}%
It follows that for any $0<$ $\overline{\delta }<\gamma $%
\begin{equation*}
\begin{array}{cc}
\underset{t,t^{\prime }\in \lbrack 0,T],t\neq t^{\prime }}{Sup}\dfrac{%
\left\vert J_{n}(t,\omega )-J_{n}(t^{\prime },\omega \right\vert _{D(B)}}{%
\left\vert t-t^{\prime }\right\vert ^{\overline{\delta }}} & \leq \underset{%
t,t^{\prime }\in \lbrack 0,T],t\neq t^{\prime }}{\sup }\mathcal{C}%
_{1}\left\vert t-t^{\prime }\right\vert ^{\gamma -1-\overline{\delta }%
}R^{1/2} \\
& \leq \mathcal{C}_{1}^{\prime }R^{1/2}.%
\end{array}%
\end{equation*}%
Hence, for $\omega \in A_{R},$ we can define%
\begin{equation*}
\begin{array}{ll}
\left\Vert J_{n}(.,\omega )\right\Vert _{C^{\overline{\delta }}([0,T],D(B))}
& =\underset{t\in \lbrack 0,T]}{\sup }\left\vert J_{n}(t,\omega )\right\vert
_{D(B)} \\
& +\underset{t,t^{\prime }\in \lbrack 0,T],t\neq t^{\prime }}{\sup }\dfrac{%
\left\vert J_{n}(t,\omega )-J_{n}(t^{\prime },\omega )\right\vert _{D(B)}}{%
\left\vert t-t^{\prime }\right\vert ^{\overline{\delta }}} \\
& \leq \mathcal{C}^{\prime \prime }R^{1/2}%
\end{array}%
\end{equation*}%
We have then for every $R>0$%
\begin{equation*}
\begin{array}{ll}
P(A_{R}) & =P\left\{ \omega \in \Lambda \
/\dint\limits_{0}^{T}\dint\limits_{0}^{T}\left( \dfrac{\left\vert
J_{n}(t)-J_{n}(t^{\prime })\right\vert _{D(B)}}{\left\vert t-t^{\prime
}\right\vert ^{\gamma }}\right) ^{2}dtdt^{\prime }\leq R\right\} \\
& \leq P\left\{ \left\Vert J_{n}(.,\omega )\right\Vert _{C^{\overline{\delta
}}([0,T],D(B))}\leq \mathcal{C}^{\prime \prime }R^{1/2}\right\}%
\end{array}%
\end{equation*}%
\newline
Let us now show that $\{J_{n}(t),t\in \lbrack 0,T]\}_{n\in
\mathbb{N}
}$ is tight on $C([0,T],D(B)).$ For $R>0$, we define%
\begin{equation*}
B_{R}=\left\{ f\in C([0,T],D(B))/\left\Vert f\right\Vert _{C^{\overline{%
\delta }}([0,T],D(B))}\leq R\right\} .
\end{equation*}%
$B_{R}$ is a compact set of $C([0,T],D(B))$ (see the Proof B1 in Appendix B)
and we have the following:
\begin{equation*}
\begin{array}{ll}
\underset{n}{\sup }P[J_{n}\in \overline{B}_{R}] & =\underset{n}{\sup }%
P[\left\Vert J_{n}(.,\omega \right\Vert _{C^{\overline{\delta }%
}([0,T],D(B))}>R] \\
& \leq \underset{n}{\text{ }\sup }P\left[ \overline{A}_{\dfrac{_{R^{2}}}{%
\mathcal{C}^{^{\prime \prime }2}}}\right] \\
& \leq \underset{n}{\text{ }\sup }P\left\{
\dint\limits_{0}^{T}\dint\limits_{0}^{T}\left( \dfrac{J\left\vert
_{n}(t)-J_{n}(t^{\prime })\right\vert _{D(B)}}{\left\vert t-t^{\prime
}\right\vert ^{\gamma }}\right) ^{2}dtdt^{\prime }>\dfrac{R^{2}}{\mathcal{C}%
^{^{\prime \prime }2}}\right\}%
\end{array}%
\end{equation*}%
Using Markov inequality, we obtain%
\begin{equation}
\underset{n}{\sup }P[J_{n}\in \overline{B}_{R}]\leq \dfrac{\mathcal{C}%
^{^{\prime \prime }2}}{R^{2}}\underset{n}{\text{ }\sup }E\left\{
\dint\limits_{0}^{T}\dint\limits_{0}^{T}\left( \frac{\left\vert
J_{n}(t)-J_{n}(t^{\prime })\right\vert _{D(B)}}{\left\vert t-t^{\prime
}\right\vert ^{\gamma }}\right) ^{2}dtdt^{\prime }\right\} .  \label{nad}
\end{equation}%
\ Since $E\left\{ \dint\limits_{0}^{T}\dint\limits_{0}^{T}\left( \dfrac{%
\left\vert J_{n}(t)-J_{n}(t^{\prime })\right\vert _{D(B)}}{\left\vert
t-t^{\prime }\right\vert ^{\gamma }}\right) ^{2}dtdt^{\prime }\right\}
<+\infty ,$ (\ref{nad}) leads to%
\begin{equation*}
\underset{n}{\sup }P[J_{n}\in \overline{B}_{R}]\leq \dfrac{M^{^{\prime
\prime }}}{R^{2}}\text{ \ }<\epsilon \text{\ }
\end{equation*}%
for $R$ large enough. This proves the tightness of $\{J_{n}(t),t\in \lbrack
0,T]\}_{n\in
\mathbb{N}
}.$ A similar argument implies that $\{I_{n}(t),t\in \lbrack 0,T]\}_{n\in
\mathbb{N}
}$ is also tight on the same space.
\end{proof}

Now, we are able to state our main result

\begin{theorem}
Let the initial condition $u_{0}\in D(B)$ be a continuous deterministic
mapping on $\Omega .$ Then, there exists a probability basis $(\overline{%
\Lambda },\overline{\mathcal{F}},(\overline{\mathcal{F}_{t}})_{t\in \lbrack
0,T]},\overline{P})$ on which there is a cylindrical Wiener process $(%
\overline{W}(t))_{t\in \lbrack 0,T]}$ and a mild solution $(u(t))_{t\in
\lbrack 0,T]}$ of (\ref{6}) in $\mathcal{C(}[0,T],D(B))$.
\end{theorem}

\begin{proof}
The tightness of $\{I_{n}(t),t\in \lbrack 0,T]\}_{n\in
\mathbb{N}
}$ and $\{J_{n}(t),t\in \lbrack 0,T]\}_{n\in
\mathbb{N}
}$ implies that $\{u_{n}(t),t\in \lbrack 0,T]\}_{n\in
\mathbb{N}
}$ converges weakly on $C([0,T],D(B)),$ this means that there exists a
subsequence $\{N_{k}\}_{k\in
\mathbb{N}
}\subset
\mathbb{N}
$ and a $C([0,T],D(B))$-valued random variable $v$ such that
\begin{equation*}
u_{N_{k}}\rightarrow v\text{ \ on }C([0,T],D(B)),\text{ as }k\rightarrow
+\infty .
\end{equation*}%
By Skorohod Representation Theorem, there exists a probability space $(%
\overline{\Lambda },\overline{\mathcal{F}},\overline{P})$ $\ $with a
filtration $(\overline{\mathcal{F}_{t}})_{t\in \lbrack 0,T]}$ and $%
C([0,T],D(B))$-valued random variables $(\widetilde{u}_{N})_{N\in
\mathbb{N}
}$ and $u$ such that as $N\rightarrow +\infty $%
\begin{equation*}
\widetilde{u}_{N}\rightarrow u\text{ \ \ \ \ \ \ \ \ }\overline{P}-a.s\text{
on }C([0,T],D(B))
\end{equation*}%
and
\begin{equation*}
\widetilde{u}_{N}\ \overset{in\text{ }law}{=u}_{N}\text{ \ and \ }\overset{in%
\text{ }law}{u=v}.
\end{equation*}%
The mild solution $u_{N}(.)$ of (\ref{approx}) is also a weakened solution
of (\ref{approx}) (see for instance \cite{CH}), that is $u_{N}(.)$ satisfies
the following integral equation
\begin{equation*}
\begin{array}{ll}
u_{N}(t) & =u_{0}+A(\dint\limits_{0}^{t}u_{N}(s)ds)-\dint\limits_{0}^{t}B
\left[ u_{N}(s)g_{G}\left( u_{N}(s)\right) \right] ds+\dint%
\limits_{0}^{t}a_{N}(u_{N}(s))dW(s).%
\end{array}%
\end{equation*}%
Hence, for each $N$,%
\begin{equation*}
\begin{array}{ll}
M_{N}(t) & =\dint\limits_{0}^{t}a_{N}(u_{N}(s))dW(s) \\
& =u_{N}(t)-u_{0}-A(\dint\limits_{0}^{t}u_{N}(s)ds)+\dint\limits_{0}^{t}B
\left[ u_{N}(s)g_{G}\left( u_{N}(s)\right) \right] ds%
\end{array}%
\end{equation*}%
is a square integrable martingale with respect to $\mathcal{F}%
_{t}^{N}=\sigma \{u_{N}(s),s\leq t\}$ and has the following quadratic
variation%
\begin{equation}
\langle M_{N}(t),M_{N}(t)\rangle
=\dint\limits_{0}^{t}a_{N}(u_{N}(s))a_{N}^{\ast }(u_{N}(s))ds.  \label{qv}
\end{equation}%
Since $\pounds (\widetilde{u}_{N})=\pounds (u_{N})$ ($\pounds $ denotes the
probability distribution)$,$%
\begin{equation*}
\widetilde{M}_{N}(t)=\widetilde{u}_{N}(t)-u_{0}-A(\dint\limits_{0}^{t}%
\widetilde{u}_{N}(s)ds)+\dint\limits_{0}^{t}B\left[ \widetilde{u}%
_{N}(s)g_{G}\left( \widetilde{u}_{N}(s)\right) \right] ds
\end{equation*}%
has the same distribution as $M_{N}(t)$ and hence
\begin{equation}
\overline{E}\left\vert \widetilde{M}_{N}(t)\right\vert
_{D(B)}^{2}=E\left\vert M_{N}(t)\right\vert _{D(B)}^{2}<+\infty
\label{samelaw}
\end{equation}%
(we denote by $\overline{E}$ the expectation with respect to the probability
measure $\overline{P}$)$.$ \newline
On another hand, since $\{M_{N}(t)\}_{t}$ is a martingale, that is
\begin{equation*}
E(M_{N}(t)-M_{N}(s)\text{ }/\mathcal{F}_{s}^{N})=0,\text{ \ }\forall \text{ }%
0\leq s<t\leq T
\end{equation*}%
and using the fact that

\begin{equation*}
M_{N}(t)-M_{N}(s)=u_{N}(t)-u_{N}(s)-A(\int\limits_{s}^{t}u_{N}(s)ds)+\dint%
\limits_{s}^{t}B\left[ u_{N}(s)g_{G}\left( u_{N}(s)\right) \right] ds
\end{equation*}%
has the same law as

\begin{equation*}
\widetilde{M}_{N}(t)-\widetilde{M}_{N}(s)=\widetilde{u}_{N}(t)-\widetilde{u}%
_{N}(s)-A(\int\limits_{s}^{t}\widetilde{u}_{N}(s)ds)+\dint\limits_{s}^{t}B%
\left[ \widetilde{u}_{N}(s)g_{G}\left( \widetilde{u}_{N}(s)\right) \right]
ds,
\end{equation*}%
it holds that
\begin{equation*}
\overline{E}(\widetilde{M}_{N}(t)-\widetilde{M}_{N}(s)\text{ }/\overline{%
\mathcal{F}}_{s}^{N})=0,\text{ \ }\forall \text{ }0\leq s<t\leq T
\end{equation*}%
which implies that $\{\widetilde{M}_{N}(t)\}_{t}$ is a\ square integrable
martingale on $(\overline{\Lambda },\overline{\mathcal{F}},\overline{P})$
with respect to $\overline{\mathcal{F}}_{t}^{N}=\sigma \{\widetilde{u}%
_{N}(s),s\leq t\}.$ \newline
On another hand, since $M_{N}(t)$ has the quadratic variation (\ref{qv}),
this means that for $x,y$ $\in D(B),$
\begin{equation*}
\langle M_{N}(t),x\rangle \langle M_{N}(t),y\rangle
-\dint\limits_{0}^{t}\langle a_{N}(u_{N}(z)),x\rangle \langle
a_{N}(u_{N}(z)),y\rangle dz\text{ is an }\mathcal{F}_{t}^{N}-\text{%
martingale.}
\end{equation*}%
Then%
\begin{equation*}
\begin{array}{c}
E(\langle M_{N}(t),x\rangle \langle M_{N}(t),y\rangle -\langle
M_{N}(s),x\rangle \langle M_{N}(s),y\rangle \\
-\dint\limits_{0}^{t}\langle a_{N}(u_{N}(z)),x\rangle \langle
a_{N}(u_{N}(z)),y\rangle dz+\dint\limits_{0}^{s}\langle
a_{N}(u_{N}(z)),x\rangle \langle a_{N}(u_{N}(z)),y\rangle dz\text{ }/%
\mathcal{F}_{s}^{N})=0.%
\end{array}%
\end{equation*}%
It follows that%
\begin{equation*}
\begin{array}{lc}
\overline{E}(\langle \widetilde{M}_{N}(t),x\rangle \langle \widetilde{M}%
_{N}(t),y\rangle -\langle \widetilde{M}_{N}(s),x\rangle \langle \widetilde{M}%
_{N}(s),y\rangle &  \\
-\dint\limits_{0}^{t}\langle a_{N}(\widetilde{u}_{N}(z)),x\rangle \langle
a_{N}(\widetilde{u}_{N}(z)),y\rangle dz+\dint\limits_{0}^{s}\langle a_{N}(%
\widetilde{u}_{N}(z)),x\rangle \langle a_{N}(\widetilde{u}_{N}(z)),y\rangle
dz\text{ }/\overline{\mathcal{F}}_{s}^{N})=0. &
\end{array}%
\end{equation*}
\newline
Hence $\langle \widetilde{M}_{N}(t),x\rangle \langle \widetilde{M}%
_{N}(t),y\rangle -\dint\limits_{0}^{t}\langle a_{N}(\widetilde{u}%
_{N}(z)),x\rangle \langle a_{N}(\widetilde{u}_{N}(z)),y\rangle dz$ is an $%
\overline{\mathcal{F}}_{s}^{N}-$martingale and as a consequence, $\widetilde{%
M}_{N}(t)$ has the unique quadratic variation process
\begin{equation*}
\langle \widetilde{M}_{N}(t),\widetilde{M}_{N}(t)\rangle
=\dint\limits_{0}^{t}a_{N}(\widetilde{u}_{N}(s))a_{N}^{\ast }(\widetilde{u}%
_{N}(s))ds.
\end{equation*}%
From (\ref{samelaw}), since

\begin{equation*}
\forall N\in
\mathbb{N}
,\underset{0\leq t\leq T}{\sup }\overline{E}\left\vert \widetilde{M}%
_{N}(t)\right\vert _{D(B)}^{2}=\underset{0\leq t\leq T}{\sup }E\left\vert
M_{N}(t)\right\vert _{D(B)}^{2}<+\infty ,
\end{equation*}%
this means that $(\widetilde{M}_{N})_{N\in
\mathbb{N}
}$ is a sequence of uniformly integrable martingales. Therefore, there
exists on $(\overline{\Lambda },\overline{\mathcal{F}},\overline{P})$ a
square integrable martingale $\widetilde{M}$ such that for $N\rightarrow
+\infty ,\quad \widetilde{M}_{N}\rightarrow M$\ and%
\begin{equation*}
M(t)=u(t)-u_{0}-A(\dint\limits_{0}^{t}u(s)ds)+\dint\limits_{0}^{t}B\left[
u(s)g_{G}\left( u(s)\right) \right] ds
\end{equation*}%
with%
\begin{equation*}
\langle M(t),M(t)\rangle =\lim_{N\rightarrow +\infty
}\dint\limits_{0}^{t}a_{N}(\widetilde{u}_{N}(s))a_{N}^{\ast }(\widetilde{u}%
_{N}(s))ds=\dint\limits_{0}^{t}\lambda u^{+}(s)ds.
\end{equation*}%
By the representation theorem for square integrable martingales \cite{Da
Prato}, there exists a cylindrical Wiener process $\overline{W}$ defined on $%
(\overline{\Lambda },\overline{\mathcal{F}},\overline{\mathcal{F}}_{t\in
\lbrack 0,T]},\overline{P})$ such that
\begin{equation*}
M(t)=\dint\limits_{0}^{t}\sqrt{\lambda u^{+}(s)}d\overline{W}(s),\text{ \ \ }%
t\in \lbrack 0,T].
\end{equation*}%
Then $u$ satisfies that
\begin{equation*}
u(t)=u_{0}+A(\dint\limits_{0}^{t}u(s)ds)-\dint\limits_{0}^{t}B\left[
u(s)g_{G}\left( u(s)\right) \right] ds+\dint\limits_{0}^{t}\sqrt{\lambda
u^{+}(s)}d\overline{W}(s),\text{ \ \ }t\in \lbrack 0,T]
\end{equation*}%
or equivalently
\begin{equation*}
\begin{array}{c}
u(t)=T(t)u_{0}-\dint_{0}^{t}T(t-s)B\left[ u(s)g_{G}\left( u(s)\right) \right]
ds \\
+\dint_{0}^{t}T(t-s)C\left( u(s)\right) d\overline{W}(s).%
\end{array}%
\end{equation*}
\end{proof}

\bigskip \textbf{Appendix A}

\begin{lemma}
\label{A1} For $\eta \in \lbrack 0,\frac{1}{2}],$ there exists constants $%
C^{\prime },C_{1}^{\prime }>0,$ such that%
\begin{equation}
\dint\limits_{0}^{s}\left\vert T(t-u)-T(s-u)\right\vert _{D(B)}^{2}du\leq
C^{\prime }\left\vert t-s\right\vert ^{\eta }\text{ \ \ \ \ \ \ }  \label{1}
\end{equation}%
and%
\begin{equation}
\dint\limits_{s}^{t}\left\vert T(t-u)\right\vert _{D(B)}^{2}du\leq
C_{1}^{\prime }\left\vert t-s\right\vert ^{\eta }\text{ }  \label{2}
\end{equation}%
with$\ $\ $0\leq s<t\leq T.$

\begin{proof}
We represent $T(t)$ in terms of the eigenvalues of $A=D\dfrac{d^{2}}{dx^{2}}$
denoted $-w_{j}^{2}$ to which correspond the eigenvectors $\varphi _{j}:$%
\newline
\begin{equation*}
T(t)=\dsum\limits_{j=1}^{\infty }\exp \{-w_{j}^{2}t\}\left\langle .,\varphi
_{j}\right\rangle \varphi _{j}
\end{equation*}%
Then, we have%
\begin{equation*}
\begin{array}{l}
\dint\limits_{0}^{s}\left\vert T(t-u)-T(s-u)\right\vert _{D(B)}^{2}du \\
\leq \dint\limits_{0}^{s}\dsum\limits_{j=1}^{\infty }\left\vert \exp
\{-w_{j}^{2}(t-u)\}-\exp \{-w_{j}^{2}(s-u)\}\right\vert
_{D(B)}^{2}\left\vert \left\langle .,\varphi _{j}\right\rangle \right\vert
_{D(B)}^{2}\left\vert \varphi _{j}\right\vert _{D(B)}^{2}du \\
\leq \dint\limits_{0}^{s}\dsum\limits_{j=1}^{\infty }\left\vert \exp
\{-w_{j}^{2}(s-u)\}\right\vert ^{2}\left\vert 1-\exp
\{-w_{j}^{2}(t-s)\}\right\vert ^{2}\left\vert \left\langle .,\varphi
_{j}\right\rangle \right\vert _{D(B)}^{2}\left\vert \varphi _{j}\right\vert
_{D(B)}^{2}du \\
\leq \dsum\limits_{j=1}^{\infty }\left\vert \left\langle .,\varphi
_{j}\right\rangle \right\vert _{D(B)}^{2}\left\vert \varphi _{j}\right\vert
_{D(B)}^{2}\left\vert w_{j}^{2}(t-s)\right\vert ^{2\lambda
}\dint\limits_{0}^{s}\exp \{-2w_{j}^{2}(s-u)\}du \\
\leq \underset{j}{\text{ }\sup }\left\vert \left\langle .,\varphi
_{j}\right\rangle \right\vert _{D(B)}^{2}\underset{j}{\sup }\left\vert
\varphi _{j}\right\vert _{D(B)}^{2}\left\vert t-s)\right\vert ^{2\lambda
}\dsum\limits_{j=1}^{\infty }w_{j}^{-(2-4\lambda )} \\
\leq C_{0}\left\vert t-s)\right\vert ^{2\lambda }\dsum\limits_{j=1}^{\infty
}w_{j}^{-(2-4\lambda )}%
\end{array}%
\end{equation*}%
where we have used the fact that for all $u\in \lbrack 0,1],$
\begin{equation*}
1-\exp \{-x\}\leq x^{u}\text{ \ \ }(x>0).
\end{equation*}%
For $\lambda \in \lbrack 0,\frac{1}{4}],$ we have $\dsum\limits_{j=1}^{%
\infty }w_{j}^{-(2-4\lambda )}<\infty .$ If we set $\eta =2\lambda $ $(\eta
\in \lbrack 0,\frac{1}{2}]),$ hence%
\begin{equation*}
\dint\limits_{0}^{s}\left\vert T(t-u)-T(s-u)\right\vert _{D(B)}^{2}du\leq
C^{\prime }\left\vert t-s)\right\vert ^{\eta }
\end{equation*}%
Similarly, for the same $\eta \in \lbrack 0,\frac{1}{2}]$ as above, we have%
\begin{equation*}
\begin{array}{l}
\dint\limits_{s}^{t}\left\vert T(t-u)\right\vert _{D(B)}^{2}du\leq
\dsum\limits_{j=1}^{\infty }\left\vert \left\langle .,\varphi
_{j}\right\rangle \right\vert _{D(B)}^{2}\left\vert \varphi _{j}\right\vert
_{D(B)}^{2}\ \dint\limits_{s}^{t}\exp \{-2w_{j}^{2}(t-u)\}du\  \\
\leq \dfrac{1}{2}\underset{j}{\sup }\left\vert \left\langle .,\varphi
_{j}\right\rangle \right\vert _{D(B)}^{2}\underset{j}{\sup }\left\vert
\varphi _{j}\right\vert _{D(B)}^{2}\dsum\limits_{j=1}^{\infty
}w_{j}^{-2}(1-\exp \{-2w_{j}^{2}(t-s)\}) \\
\leq C_{0}^{\prime }\left\vert t-s\right\vert ^{\eta
}\dsum\limits_{j=1}^{\infty }w_{j}^{-(2-2\eta )} \\
\leq C_{1}^{\prime }\left\vert t-s)\right\vert ^{\eta }%
\end{array}%
\end{equation*}
\end{proof}
\end{lemma}

\begin{lemma}
\textbf{\label{A2}}\bigskip\ For $\eta \in \lbrack 0,\frac{1}{2}],$ there
exists constants $C^{\prime \prime },C_{1}^{^{\prime \prime }}>0,$such that%
\begin{equation}
\dint\limits_{0}^{s}\left\vert B\left[ T(t-u)-T(s-u)\right] \right\vert
_{D(B)}^{2}du\leq C^{\prime \prime }\left\vert t-s\right\vert ^{\eta }\text{
\ \ \ \ \ \ }
\end{equation}%
\begin{equation}
\dint\limits_{s}^{t}\left\vert BT(t-u)\right\vert _{D(B)}^{2}du\leq
C_{1}^{^{\prime \prime }}\left\vert t-s\right\vert ^{\eta }\text{ }
\end{equation}%
with$\ 0\leq s<t\leq T.$
\end{lemma}

\begin{proof}
From the representation%
\begin{equation*}
T(t)=\dsum\limits_{j=1}^{\infty }\exp \{-w_{j}^{2}t\}\left\langle .,\varphi
_{j}\right\rangle \varphi _{j}.
\end{equation*}%
we obtain
\begin{equation*}
BT(t)=\dsum\limits_{j=1}^{\infty }\exp \{-w_{j}^{2}t\}\left\langle .,\varphi
_{j}\right\rangle B\varphi _{j}.
\end{equation*}%
Hence%
\begin{equation*}
\begin{array}{lc}
\dint\limits_{0}^{s}\left\vert B\left[ T(t-u)-T(s-u)\right] \right\vert
_{D(B)}^{2}du &  \\
\leq \dint\limits_{0}^{s}\dsum\limits_{j=1}^{\infty }\left\vert \exp
\{-w_{j}^{2}(t-u)\}-\exp \{-w_{j}^{2}(s-u)\}\right\vert
_{D(B)}^{2}\left\vert \left\langle .,\varphi _{j}\right\rangle \right\vert
_{D(B)}^{2}\left\vert B\varphi _{j}\right\vert _{D(B)}^{2}du &  \\
\leq \dint\limits_{0}^{s}\dsum\limits_{j=1}^{\infty }\left\vert \exp
\{-w_{j}^{2}(s-u)\}\right\vert ^{2}\left\vert 1-\exp
\{-w_{j}^{2}(t-s)\}\right\vert ^{2}\left\vert \left\langle .,\varphi
_{j}\right\rangle \right\vert _{D(B)}^{2}\left\vert B\varphi _{j}\right\vert
_{D(B)}^{2}du &  \\
\leq \dsum\limits_{j=1}^{\infty }\left\vert \left\langle .,\varphi
_{j}\right\rangle \right\vert _{D(B)}^{2}\left\vert B\varphi _{j}\right\vert
_{D(B)}^{2}\left\vert w_{j}^{2}(t-s)\right\vert ^{2\lambda
}\dint\limits_{0}^{s}\exp \{-2w_{j}^{2}(s-u)\}du &  \\
\leq \underset{j}{\sup }\left\vert \left\langle .,\varphi _{j}\right\rangle
\right\vert _{D(B)}^{2}\underset{j}{\sup }\left\vert B\varphi
_{j}\right\vert _{D(B)}^{2}\left\vert t-s)\right\vert ^{2\lambda
}\dsum\limits_{j=1}^{\infty }w_{j}^{-(2-4\lambda )} &  \\
\leq C_{2}\left\vert t-s)\right\vert ^{2\lambda }\dsum\limits_{j=1}^{\infty
}w_{j}^{-(2-4\lambda )} &
\end{array}%
\end{equation*}
where we have used the fact that for all $u\in \lbrack 0,1],$
\begin{equation*}
1-\exp \{-x\}\leq x^{u}\text{ \ \ }(x>0).
\end{equation*}%
For $\lambda \in \lbrack 0,\frac{1}{4}],$ we have $\dsum\limits_{j=1}^{%
\infty }w_{j}^{-(2-4\lambda )}<\infty .$ If we set $\eta =2\lambda $ $(\eta
\in \lbrack 0,\frac{1}{2}]),$ hence%
\begin{equation*}
\dint\limits_{0}^{s}\left\vert B\left[ T(t-u)-T(s-u)\right] \right\vert
_{D(B)}^{2}du\leq C_{1}\left\vert t-s)\right\vert ^{\eta }.
\end{equation*}%
Similarly, for the same $\eta \in \lbrack 0,\frac{1}{2}]$ as above, we have%
\begin{equation*}
\begin{array}{l}
\dint\limits_{s}^{t}\left\vert BT(t-u)\right\vert _{D(B)}^{2}du\leq
\dsum\limits_{j=1}^{\infty }\left\vert \left\langle .,\varphi
_{j}\right\rangle \right\vert _{D(B)}^{2}\left\vert B\varphi _{j}\right\vert
_{D(B)}^{2}\ \dint\limits_{s}^{t}\exp \{-2w_{j}^{2}(t-u)\}du\  \\
\leq \underset{j}{\frac{1}{2}\sup }\left\vert \left\langle .,\varphi
_{j}\right\rangle \right\vert _{D(B)}^{2}\underset{j}{\sup }\left\vert
B\varphi _{j}\right\vert _{D(B)}^{2}\dsum\limits_{j=1}^{\infty
}w_{j}^{-2}(1-\exp \{-2w_{j}^{2}(t-s)\}) \\
\leq C^{\prime }\left\vert t-s\right\vert ^{\eta }\dsum\limits_{j=1}^{\infty
}w_{j}^{-(2-2\eta )} \\
\leq C_{1}^{^{\prime }}\left\vert t-s)\right\vert ^{\eta }%
\end{array}%
\end{equation*}
\end{proof}

\begin{lemma}
\label{A3}(Garsia-Rodemich-Rumsey) Let the function $\psi :[0,\infty \lbrack
\rightarrow \lbrack 0,\infty \lbrack $ non-decreasing with $\underset{%
u\rightarrow +\infty }{lim\ \psi (u)}=+\infty .$ Let the function $%
p:[0,1]\rightarrow \lbrack 0,1]$ be continuous and non-decreasing with $%
p(0)=0.$ \newline
Set $\left\{ \psi ^{-1}(u)=\underset{\psi (v)\leq u}{\sup v}\text{ if }\psi
(0)\leq u<\infty \right. .$ Let $f$ be a continuous function on $[0,1]$ and
suppose that
\begin{equation*}
\dint\limits_{0}^{1}\dint\limits_{0}^{1}\psi (\frac{\left\vert
f(x)-f(y)\right\vert }{p(x-y)})dxdy\leq B<\infty
\end{equation*}%
Then, for all $x,y\in \lbrack 0,1]$, we have%
\begin{equation*}
\left\vert f(x)-f(y)\right\vert \leq 8\dint\limits_{0}^{\left\vert
y-x\right\vert }\psi ^{-1}(\frac{4B}{u^{2}})dp(u).
\end{equation*}
\end{lemma}

\textbf{Appendix B}

\textbf{Proof B1 \ } Since
\begin{equation*}
\left\Vert f\right\Vert _{C^{\overline{\delta }}([0,T],D(B))}=\underset{t\in
\lbrack 0,T]}{Sup}\left\vert f(t)\right\vert _{D(B)}+\underset{t,t^{\prime
}\in \lbrack 0,T],t\neq t^{\prime }}{Sup}\dfrac{\left\vert f(t)-f(t^{\prime
})\right\vert _{D(B)}}{\left\vert t-t^{\prime }\right\vert ^{\overline{%
\delta }}},
\end{equation*}
this means that $\forall $ $t,t^{\prime }\in \lbrack 0,T],t\neq t^{\prime }$
\begin{equation*}
\left\vert f(t)-f(t^{\prime })\right\vert _{D(B)}\leq \left\Vert
f\right\Vert _{C^{\overline{\delta }}([0,T],D(B))}\left\vert t-t^{\prime
}\right\vert ^{\overline{\delta }}
\end{equation*}%
As $f\in B_{R},$ we have%
\begin{equation*}
\left\vert f(t)-f(t^{\prime })\right\vert _{D(B)}\leq R\left\vert
t-t^{\prime }\right\vert ^{\overline{\delta }}
\end{equation*}%
This means that $B_{R}$ is uniformly equicontinuous. By the theorem of
Ascoli-Arzela, $B_{R}$ is a compact of $C([0,T],D(B)).$

\end{document}